\documentclass[a4paper,12pt]{amsart}
\usepackage{amsmath,amssymb,amsthm}
\usepackage{array,enumerate}
\usepackage[pdftitle={Haagerup property for C*-algebras and rigidity of C*-algebras with property (T)},
pdfauthor={Yuhei Suzuki},
pdfkeywords={C*-algebras; Haagerup property; property (T); rigidity}]
{hyperref}
\newtheorem{Thm}{Theorem}[section]
\newtheorem{Prop}[Thm]{Proposition}
\newtheorem{Lem}[Thm]{Lemma}
\newtheorem{Cor}[Thm]{Corollary}
\newtheorem{Thmint}{Theorem}

\theoremstyle{definition}
\newtheorem{Rem}[Thm]{Remark}
\newtheorem{Def}[Thm]{Definition}
\newtheorem{Exm}[Thm]{Example}
\newtheorem{Exms}[Thm]{Examples}
\newtheorem{Ques}[Thm]{Question}
\newcommand{\Cs}{\mbox{${\rm C}^\ast$}}
\newcommand{\SL}{{\rm SL}}
\begin{document}
\title[Haagerup property and property (T) for \Cs -algebras]{Haagerup property for \Cs -algebras and rigidity of \Cs -algebras with property (T)}
\author{Yuhei Suzuki}
\date{}
\address{Department of Mathematical Sciences,
University of Tokyo, Komaba, Tokyo, 153-8914, Japan}
\email{suzukiyu@ms.u-tokyo.ac.jp}
\begin{abstract}
We study the Haagerup property for ${\rm C}^\ast$-algebras.
We first give new examples of ${\rm C}^\ast$-algebras with the Haagerup property.
A nuclear ${\rm C}^\ast$-algebra with a faithful tracial state always has the Haagerup property,
and the permanence of the Haagerup property for ${\rm C}^\ast$-algebras is established.
As a consequence, the class of all ${\rm C}^\ast$-algebras with the Haagerup property turns out to be quite large.
We then apply Popa's results and show the ${\rm C}^\ast$-algebras with property~{\rm (T)} have a certain rigidity property.
Unlike the case of von Neumann algebras, for the reduced group ${\rm C}^\ast$-algebras of groups with relative property~{\rm (T)},
the rigidity property strongly fails in general. 
Nevertheless, for some groups without nontrivial property~{\rm (T)} subgroups, we show a rigidity property in some cases.
As examples, we prove the reduced group ${\rm C}^\ast$-algebras of the (non-amenable) affine groups of the affine planes have a rigidity property.
\end{abstract}
\maketitle
\tableofcontents

\section{Introduction}
The Haagerup property was first defined for groups, by Haagerup \cite{Haa}, as a weakened version of amenability.
This concept is generalized to one in the context of von Neumann algebras by Choda \cite{Cho} for distinguishing particular group von Neumann algebras.
After she introduced the definition, it has been studied by many authors, for example, in \cite{CJ}, \cite{Jol}, \cite{Pop}, and \cite{Rob}.
Recently, Dong introduced a notion of the Haagerup property for a pair of a \Cs -algebra and its faithful tracial state,
by imitating the case of von Neumann algebras.

In this paper, we first give new examples of \Cs -algebras with the Haagerup property:
Every unital nuclear \Cs -algebra has the Haagerup property with respect to an arbitrary faithful tracial state.
At the same time, we also establish permanence properties of the Haagerup property.
As a consequence of these two results, we have many new examples of \Cs -algebras with the Haagerup property. Here we state these theorems.
Proofs are given in Section \ref{sec:Ex}.
\begin{Thmint}[Theorem \ref{Thm:nuc}]\label{Thmint:nuc}
Let $(A,\tau)$ be a pair of a unital nuclear \Cs -algebra and a faithful tracial state.
Then it has the Haagerup property.
\end{Thmint}
\begin{Thmint}[Theorem \ref{Thm:per}]\label{Thmint:per}
Let ${(A_i,\tau_i)}_{i\in I}$ be a family of \Cs -algebras with the Haagerup property indexed by a set $I$.
Then the following hold.
\begin{enumerate}[\upshape (1)]
\item If I is countable, then the direct product $(\prod_{i\in I} A_i,\tau)$ has the Haagerup property for any tracial state $\tau$ of the form $\tau=\sum_{i\in I}c_i\tau_i$,
where $(c_i)_{i\in I}$ is a family of positive numbers whose sum is $1$.
More generally, any \Cs -subalgebra of $(\prod_{i\in I} A_i,\tau)$ which contains both $\bigoplus_{i\in I} A_i$ and $1$ has the Haagerup property with respect to the restriction tracial state.
\item The spatial tensor product $(\bigotimes_{i\in I}A_i,\bigotimes_{i\in I}\tau_i)$ has the Haagerup property.
\item The reduced free product $(A, \tau)=$ \lower0.25ex\hbox{\LARGE $\ast$}${}_{i\in I}(A_i,\tau_i)$ has the Haagerup property.
\end{enumerate}
\end{Thmint}
The second theorem can be shown by the same proof as Jolissaint's one for von Neumann algebras \cite{Jol}.

In the second part of this paper, we give an application of the Haagerup property for \Cs -algebras.
Applying Popa's result, we have the following rigidity theorem.
\begin{Thmint}[Theorem \ref{Thm:T}]\label{Thmint:T}
Let $A$ be a \Cs -algebra which has a faithful tracial state with the Haagerup property.
Let $B$ be a \Cs -subalgebra of $A$ such that the pair $(A, B)$ has relative property~{\rm (T)}.
Then B must be residually finite dimensional.
\end{Thmint}
Applying this theorem with Leung-Ng's theorem \cite[Proposition 5.4]{LN}, we revisit the rigidity theorem of Robertson \cite[Theorem C]{Rob}.
(See also Remark \ref{Rem:Nem}.)
\begin{Thmint}[Corollary \ref{Cor:Nem}]
Let $\Gamma$ be a property~{\rm (T)} group, $A$ be a \Cs -algebra which has a faithful tracial state with the Haagerup property.
Then any unitary representation of $\Gamma$ on $A$ is weakly equivalent
to a direct sum of finite dimensional representations.
In particular, if $\Gamma$ is an infinite property~{\rm (T)} group,
then there is no nonzero $\ast$-homomorphism
from the reduced group \Cs -algebra ${\rm C}_r^*(\Gamma)$ into $A$.
\end{Thmint}

We also show this is not true for a general non-Haagerup group, even if the group has relative property~{\rm (T)} with respect to an infinite subgroup.
As an example, we give the following embeddings.
\begin{Thmint}[Theorem \ref{Thm:SL2}, Remark \ref{Rem:SL2}]\label{Thm:emb}
There are \Cs -algebras $A$, $B$ and $C$, each of which admits a faithful tracial state with the Haagerup property, having the following embeddings
\[{\rm C}_r^*(\mathbb{Z}^2\rtimes\SL _2({\mathbb{Z}})) \hookrightarrow A,\]
\[{\rm C}_r^*(H_3(\mathbb{Z})\rtimes\SL _2({\mathbb{Z}})) \hookrightarrow B,\]
\[{\rm C}_r^*(\mathbb{F}_p[t]^2\rtimes\SL _2(\mathbb{F}_p[t])) \hookrightarrow C.\]
\end{Thmint}
Note that all the pairs $\left(\mathbb{Z}^2\rtimes\SL _2({\mathbb{Z}}), \mathbb{Z}^2\right)$,
$\left(H_3(\mathbb{Z})\rtimes\SL _2({\mathbb{Z}}), H_3(\mathbb{Z})\right)$, and\\
$\left(\mathbb{F}_p[t]^2\rtimes\SL _2(\mathbb{F}_p[t]), \mathbb{F}_p[t]^2\right)$ have relative property~{\rm (T)}.
Hence Theorem \ref{Thm:emb} shows the rigidity theorem strongly fails for a general non-Haagerup group, even if it has relative property~{\rm (T)} with respect to an infinite subgroup.

However, we show a rigidity property for a group without infinite property~{\rm (T)} subgroups in some cases.
As an example, we show the following rigidity property of the affine groups of the affine planes.

\begin{Thmint}[Theorem \ref{Thm:Aff}]\label{Thmint:Aff}
Let $\mathbb{K}$ be a field which is not an algebraic extension of a finite field.
Then the reduced group \Cs -algebra ${\rm C}_r^*(\mathbb{K}^2\rtimes\SL _2({\mathbb{K}}))$
cannot embed into any \Cs -algebra which has a faithful tracial state with the Haagerup property.
\end{Thmint}

\noindent\underbar{\bf Underlying assumptions.}\\
In this paper, the following are always assumed.
\begin{itemize}
\item We always consider the topology of a group as the discrete one. 
\item We always assume a representation of a discrete group (on a Hilbert space or into an operator algebra) is a unitary one.
\item Positive definite functions on groups are always assumed to be normalized
(i.e., take value $1$ at the unit).
\end{itemize}

\noindent\underbar{\bf Notations.}\\
Here we fix notations which are used throughout in this paper.
\begin{itemize}
\item The symbol $\mathbb{M}_n$ means the matrix algebra of the size $n$ over $\mathbb{C}$.
\item For two \Cs -algebras $A$ and $B$, $A\otimes B$, $A\odot B$ mean the spatial tensor product, the algebraic tensor product, respectively.
\item For two \Cs -algebras $A$ and $B$, CP$(A,$~$B)$ means the set of all completely positive maps from $A$ to $B$.
\item For a \Cs -algebra $A$, $A^{\rm op}$ denotes the opposite algebra of $A$.
\item For a \Cs -algebra $A$ and a state $\varphi$, denote by $L^2(A, \varphi)$ the GNS-space of $(A, \varphi)$ and
denote by $\|\cdot\|_\varphi$ the norm on $L^2(A, \varphi)$.
\item For a discrete group $\Gamma$, $c_c(\Gamma)$ denotes the space of all complex valued functions on $\Gamma$ with finite supports.
\item For a discrete group $\Gamma$ and a positive definite function $\phi$ on $\Gamma$, $l_\phi^2(\Gamma)$
denotes the GNS-space of $\phi$ and $\langle\ ,\ \rangle_\phi$ denotes the inner product of $l_\phi^2(\Gamma)$.
If $\phi=\delta_e$, then as usual, we simply denote $l_{\delta_e}^2(\Gamma)$ by $l^2(\Gamma)$ and $\langle\ ,\ \rangle_{\delta_e}$ by $\langle\ ,\ \rangle_2$.
\item With above notations, for $g\in\Gamma$, the canonical image of $\delta_g$ in $l_\phi^2(\Gamma)$ is denoted by $\delta _g ^\phi$.
\item For a set $X$ and a subset $S\subset X$, we denote the characteristic function of $S$ on $X$
by $\chi_S$.
\item The finite field with order $p$ is denoted by $\mathbb{F}_p$.
\item The terms {\rm u.c.p.\null}, {\rm c.c.p.\null}, {\rm c.p.\null}
are the abbreviations of ``Unital Completely Positive'', ``Contractive Completely Positive'', ``Completely Positive'', respectively.
\end{itemize}

\section{Preliminaries}

Recall that a discrete group $\Gamma$ is said to have the Haagerup property (also known as Gromov's a-T-menability)
if there is a net $(\phi_n)_n$ of positive definite functions each of which vanishes at infinity and the net converges to $1$ pointwise.
It is well-known if we replace the condition ``vanish at infinity'' by ``have a finite support'', then this is equivalent to amenability (see \cite[Theorem 2.6.8]{BO}).
In this sense, the Haagerup property is considered as a weak version of amenability.
Amenability of groups is quite useful but a strong condition, so many important groups fail to have amenability.
On the other hand, the Haagerup property, a weak version of amenability, is satisfied by many important non-amenable groups, for example, the free groups, the Coxeter groups, and so on.
Moreover, in many applications, the Haagerup property is sufficiently useful.
For example, for the groups with the Haagerup property, the Baum-Connes conjecture is true,
and consequently many important conjectures (e.g., the Novikov conjecture, the Kaplansky conjecture, etc.) are also true.
Moreover, the groups with the Haagerup property do not have (relative) property~{\rm (T)}, which is a rigidity property of discrete groups.
In many situations, a group with (relative) property~{\rm (T)} is essentially hard to study,
so at least for the groups with the Haagerup property, essential difficulties would not arise.
For these reasons, it is interesting to study the Haagerup property.

As in the case of amenability, the Haagerup property also has many characterizations, but the above form is the most suitable formulation for extending it to the setting of operator algebras.
For more information about the Haagerup property for discrete (or more generally, for locally compact) groups, we refer the reader to the book \cite{Che}.

Recently, Dong \cite{Don} gave a definition of the Haagerup property for a pair of a unital \Cs -algebra and its faithful tracial state as follows.
\begin{Def}\label{Def:Haa}
Let $A$ be a unital \Cs -algebra, $\tau$ a faithful tracial state on $A$.
The pair $(A, \tau)$ is said to have the {\it Haagerup property} if
there is a net $(\phi_i)_{i\in I}$ of {\rm u.c.p.\null} maps from $A$ to itself satisfying the following conditions.
\begin{enumerate}[\upshape (1)]
\item Each $\phi_i$ decreases $\tau$; i.e., for any positive element $a \in A$,
we have $\tau(\phi_i(a))\leq\tau(a)$.
\item For any $a \in A, \|\phi_i(a)-a\|_\tau$ converges to $0$ as $i$ tends to infinity.
\item Each $\phi_i$ is $L^2$-compact; i.e., from the first condition, $\phi_i$ extends to a bounded operator on its GNS-space $L^2(A,\tau)$,
which is compact.
\end{enumerate}
For brevity, we sometimes say $\tau$ is a tracial state with the Haagerup property.
\end{Def}
This is a straightforward generalization of the definition of the Haagerup property for von Neumann algebras, which has been introduced in \cite{Cho}.
\begin{Rem}
The definition of the Haagerup property for von Neumann algebras is the same as above
except for the additional assumption that the tracial state is normal.
With this definition, Jolissaint \cite{Jol} proves the Haagerup property of von Neumann algebras
does not depend on the choice of faithful normal tracial states.
However, as we prove in Section \ref{sec:App} (Theorem \ref{Thm:trd}), in the case of \Cs -algebras,
this is no longer true. 
\end{Rem}
\begin{Rem}
As in the case of von Neumann algebras \cite[Remark 12.1.17]{BO},
the condition of $\phi_i$ being {\rm u.c.p.\null} can be relaxed by $\phi_i$ being {\rm c.c.p.\null},
and we can take $\phi_i$ so that it preserves $\tau$.
This is done by replacing $(\phi_i)_{i\in I}$ by the net
\[\psi_{i,n}(x):=c_n\phi_i(x)+\frac{(\tau-c_n\tau\circ\phi_i)(x)}{(1-c_n\tau\circ\phi_i)(1)}(1-c_n\phi_i(1))\]
on $I\times\mathbb{N}$, where $c_n:=1-1/n$.
\end{Rem}
Here we recall the fundamental properties of the Haagerup property for \Cs -algebras.
\begin{Thm}[Dong \cite{Don}]\label{Thm:Don}
The following hold.
\begin{enumerate}[\upshape (1)]
\item For a discrete group $\Gamma$, it has the Haagerup property if and only if
the reduced group \Cs -algebra ${\rm C}_r^*(\Gamma)$ of $\Gamma$ has the Haagerup property with respect to the canonical tracial state.
\item The Haagerup property for \Cs -algebras is closed under taking the reduced crossed product by a trace-preserving action of an amenable group.
\item For any finite dimensional \Cs -algebra $A$ and a group $\Gamma$ with the Haagerup property acting on $A$,
the reduced crossed product of $A$ by $\Gamma$ has the Haagerup property with respect to the canonical extension of any $\Gamma$-invariant faithful tracial state on $A$, which always exists.
\end{enumerate}
\end{Thm}
Note that part (3) above is not mentioned in the paper of Dong,
but this immediately follows from his study of the relative Haagerup property \cite[Section 3]{Don}.

For later applications, we recall the definition of relative property~{\rm (T)}.
Relative property~{\rm (T)} is a rigidity property of groups, which is negated by the Haagerup property in the following sense:
A group which has relative property~{\rm (T)} with respect to an infinite subgroup does not have the Haagerup property.
For more about property~{\rm (T)}, we refer the reader to the book \cite{BHV} of Bekka, de la Harpe, and Valette.
\begin{Def}
Let $\Gamma$ be a group and let $\Lambda$ be a subgroup of $\Gamma$.
The pair $(\Gamma, \Lambda)$ is said to have {\it relative property~(T)} if
any net $(\phi_n)$ of positive definite functions on $\Gamma$ that converges to $1$ pointwise converges uniformly on $\Lambda$.
A group $\Gamma$ is said to have {\it property~(T)} if the pair $(\Gamma,\Gamma)$ has relative property~{\rm (T)}.
\end{Def}
Here we review examples of relative property~{\rm (T)} groups.
\begin{Exms}[\cite{BHV}]\ 
\begin{itemize}
\item For $n\geq 3$, $\SL _n({\mathbb{Z}})$ has property~{\rm (T)}.
\item The pair $\left(\mathbb{Z}^2\rtimes\SL _2({\mathbb{Z}}), \mathbb{Z}^2\right)$ has relative property~{\rm (T)}.
\item The pair $\left(H_3(\mathbb{Z})\rtimes\SL _2(\mathbb{Z}), H_3(\mathbb{Z})\right)$ has relative property~{\rm (T)}.
\item The pair $\left(\mathbb{F}_p[t]^2\rtimes\SL _2(\mathbb{F}_p[t]), \mathbb{F}_p[t]^2\right)$ has relative property~{\rm (T)}.
\item For $n\geq 2$, $Sp_{2n}({\mathbb{Z}})$ has property~{\rm (T)}.
\end{itemize}
\end{Exms}
Property~{\rm (T)} has been also defined for operator algebras.
For von Neumann algebras, it was first done by Connes-Jones \cite{CJ} for type {\rm I\hspace{-.1em}I$_1$} factors,
then it was extended to general finite von Neumann algebras by Popa \cite{Pop}.
For \Cs -algebras, property~{\rm (T)} was first introduced by Bekka \cite{Bek},
then its (formally) strengthened version, called strong property~{\rm (T)} was introduced by Leung-Ng \cite{LN}.
The notion of strong property~{\rm (T)} is closer to that of Popa's property~{\rm (T)} than Bekka's one, which is suitable for our application.
Moreover, most important examples of \Cs -algebras having Bekka's property~{\rm (T)}
also have strong property~{\rm (T)}. Actually, it is not known whether these two notions coincide or not \cite[page 3057]{LN}.
For these reasons, we only use Leung-Ng's notion and we simply call it property~{\rm (T)} instead of strong property~{\rm (T)}.
To recall these definitions we first need the notion of a Hilbert bimodule.
\begin{Def}
Let $A$ be a \Cs -algebra or a von Neumann algebra.
A Hilbert space $H$ is called a {\it Hilbert} $A$-{\it bimodule}
if it is equipped with the actions $\pi$ of $A$ and $\rho$ of $A^{\rm op}$
such that these two actions are mutually commuting.
When $A$ is a von Neumann algebra, then we further assume both actions are normal.
We refer to $\pi$, $\rho$ as the left, right action of $A$, respectively,
and use the notation
\[x\xi y:=\pi(x)\rho(y^{\rm op})\xi\]
for $x, y\in A$ and $\xi\in H$.
\end{Def}
\begin{Def}[Leung-Ng \cite{LN}]
Let $A$ be a \Cs -algebra, $B$ be a \Cs -subalgebra of $A$.
The pair $(A,B)$ is said to have {\it relative property~(T)} if for any $\epsilon >0$, there exist $\delta >0$ and a finite subset $Q\subset A$, such that the following holds:
For any Hilbert $A$-bimodule $H$ and a unit vector $\xi\in H$ satisfying
$\|x\xi-\xi x\|<\delta$ for all $x\in Q$, there exists a $B$-central unit vector $\xi _0$ (i.e., $x\xi_0=\xi_0 x$ holds for all $x\in B$) with $\|\xi _0-\xi\|<\epsilon$.
If the pair $(A,A)$ has relative property~{\rm (T)}, then $A$ is said to have {\it property~(T)}.
\end{Def}
\begin{Def}[Popa \cite{Pop}]
Let $M$ be a finite von Neumann algebra with a faithful normal tracial state,
$B$ be a von Neumann subalgebra of $M$.
The pair $(M,B)$ is said to have {\it relative property~(T)}
if there is a faithful normal tracial state $\tau$ on $M$ satisfying the following condition:
For any $\epsilon >0$, there exist $\delta >0$ and a finite subset $Q\subset M$ such that
if $H$ is a Hilbert $M$-bimodule, $\xi$ is a unit vector in $H$
satisfying $\|\langle\cdot\xi, \xi\rangle -\tau\|<\delta, \|\langle\xi\cdot, \xi\rangle -\tau\|<\delta$ and $\|x\xi -\xi x\| <\delta$,
then there is a $B$-central unit vector $\xi _0$ with $\|\xi _0-\xi\|<\epsilon$.
If the pair $(M, M)$ has relative property~{\rm (T)}, then $M$ is said to have {\it property~(T)}.
\end{Def}
Here we recall the basic facts of property~{\rm (T)} for operator algebras.
The next theorem says the definition is natural.
\begin{Thm}[Leung-Ng \cite{LN}, Popa \cite{Pop}]
Let $(\Gamma,\Lambda)$ be a pair of a group and its subgroup.
Then the following are equivalent.
\begin{enumerate}[\upshape (1)]
\item The pair $\left(\Gamma,\Lambda\right)$ has relative property~{\rm (T)}.
\item The pair $\left({\rm C}_r^*(\Gamma),{\rm C}_r^*(\Lambda)\right)$ has relative property~{\rm (T)}.
\item The pair $\left({\rm C}^*(\Gamma),{\rm C}^*(\Lambda)\right)$ has relative property~{\rm (T)}.
\item The pair $\left(L(\Gamma),L(\Lambda)\right)$ has relative property~{\rm (T)}.
\end{enumerate}
\end{Thm}
For the proof, we refer the reader to \cite{LN}, \cite{Pop}.

Next we recall the permanence properties of property~{\rm (T)} for \Cs -algebras.
This gives many examples of \Cs -algebras with property~{\rm (T)}.
\begin{Prop}[Leung-Ng \cite{LN}]\label{Prop:PerT}
Property~{\rm (T)} for \Cs -algebras is preserved by the following operations.
\begin{enumerate}[\upshape (1)]
\item Taking a quotient.
\item Taking a maximal {\rm (}hence arbitrary{\rm )} tensor product.
\item Taking a full {\rm (}hence arbitrary{\rm )} crossed product by a property~{\rm (T)} group.
\end{enumerate}
\end{Prop}

For further background knowledge of this paper,
the book \cite{BO} of Brown and Ozawa is a good reference.

\section{\texorpdfstring{Examples of \Cs -algebras with the Haagerup Property}{Examples of C*-algebras with the Haagerup Property}}\label{sec:Ex}

The goal of this section is to prove Theorems \ref{Thmint:nuc} and \ref{Thmint:per}.
As the study of nuclearity of \Cs -algebras (e.g., a proof of the fact that nuclearity passes to a quotient \cite[Chapter 9]{BO}),
in order to prove Theorem \ref{Thmint:nuc}, we need a deep theorem from the theory of von Neumann algebras.
To state and apply the theorem of Connes below, we need the following concepts of von Neumann algebras.
\begin{Def}
Let $M$ be a von Neumann algebra.
It is said to be {\it injective} if for any unital \Cs -algebra $A$ and for any its closed self-adjoint subspace $N$ containing the unit of $A$,
any {\rm u.c.p.\null} map from $N$ to $M$ can be extended to a {\rm u.c.p.\null} map from $A$ to $M$.
\end{Def}
\begin{Rem}\label{Rem:Arv}
By Arveson's extension theorem \cite[Theorem 1.6.1]{BO},
injectivity of a von Neumann algebra $M$ is equivalent to the
existence of a conditional expectation from $B(H)$ onto $M$ for some faithful representation $M\subset B(H)$ of $M$ on a Hilbert space $H$.
\end{Rem}
\begin{Def}
Let $M$ be a von Neumann algebra with the separable predual.
It is said to be {\it AFD} (abbreviation of ``Approximately Finite Dimensional'') if there is an increasing sequence of finite dimensional
$\ast$-subalgebras of $M$ whose union is dense in $M$ in the strong operator topology.
\end{Def}
The following theorem of Connes states these two properties are equivalent in the separable case.
\begin{Thm}[Connes \cite{Con}]
For a von Neumann algebra $M$ with the separable predual, the following are equivalent.
\begin{enumerate}[\upshape (1)]
\item The von Neumann algebra $M$ is injective.
\item The von Neumann algebra $M$ is AFD.
\end{enumerate}
\end{Thm}
For the proof and more information about the theorem, we refer the reader to \cite{Con}.
Before proving our theorem, we need a lemma.
\begin{Lem}\label{Lem:CE}
Let $M$ be a {\rm (}not necessarily separable{\rm )} injective von Neumann algebra with a faithful normal tracial state $\tau$.
Then there exists a net $(\Phi_n)_n$ of conditional expectations on $M$
which satisfies the following three conditions.
\begin{enumerate}[\upshape (1)]
\item Each image of $\Phi_n$ is finite dimensional.
\item Each $\Phi_n$ preserves $\tau$.
\item The net $(\Phi_n)_n$ converges to the identity map in the pointwise strong operator topology.
\end{enumerate}
\end{Lem}
\begin{proof}
Note first that since $M$ has a faithful normal tracial state,
\cite[Lemma 1.5.11]{BO} (with Remark \ref{Rem:Arv}) shows each von Neumann subalgebra of $M$ is injective.
From this, for each finite subset $\mathfrak{F}$ of $M$,
the von Neumann subalgebra $W^*(\mathfrak{F})$ of $M$ generated by $\mathfrak{F}$ is injective and separable.
From this, by Connes's theorem, each von Neumann algebra $W^*(\mathfrak{F})$ is an AFD von Neumann algebra.
Consequently, for each finite subset $\mathfrak{F}$ of $M$, there is a finite dimensional $\ast$-subalgebra $M_{\mathfrak{F}}$ of $M$, such that
${\rm dist}_\tau(x, M_\mathfrak{F})<1/|\mathfrak{F}|$ for all $x\in\mathfrak{F}$ (where $|\mathfrak{F}|$ is the cardinality of $\mathfrak{F}$).
Then again by \cite[Lemma 1.5.11]{BO}, for each finite subset $\mathfrak{F}$ of $M$,
there is a $\tau$-preserving conditional expectation $E_\mathfrak{F}$
from $M$ onto $M_\mathfrak{F}$.
Then notice that for any $\tau$-preserving conditional expectation $E$
with the range $N$,
we have
\[\| x-E(x)\|_\tau ={\rm dist}_\tau(x, N).\]
On the other hand, for each $x\in M$,
${\rm dist}_\tau\left(x, M_\mathfrak{F}\right)$ converges to zero as $\mathfrak{F}$ tends to infinity.
From this, the net $(E_\mathfrak{F})_\mathfrak{F}$ satisfies the desired three conditions.
\end{proof}

Now we prove our first main result, Theorem \ref{Thmint:nuc}.
\begin{Thm}\label{Thm:nuc}
Let $(A,\tau)$ be a pair of a unital nuclear \Cs -algebra and a faithful tracial state.
Then it has the Haagerup property.
\end{Thm}
\begin{proof}
Let $A$ be a unital nuclear \Cs -algebra, $\tau$ be a faithful tracial state on $A$.
We will show the pair $(A, \tau)$ has the Haagerup property.
Let $\pi_\tau$ be the GNS-representation of $\tau$.
Then, since $A$ is nuclear, it is easy to show $\pi_\tau(A)''$ is an injective von Neumann algebra.
Using Lemma \ref{Lem:CE}, we can choose a net $(\Phi_\alpha)_\alpha$ of conditional expectations on $\pi_\tau(A)''$
satisfying the following three conditions.
\begin{itemize}
\item Each image of $\Phi_\alpha$ has a finite dimension.
\item Each $\Phi_\alpha$ preserves $\tau$.
\item The net $(\Phi_\alpha)_\alpha$ converges to the identity map in the pointwise strong operator topology.
\end{itemize}
From these conditions, each $\Phi_\alpha$ is $L^2$-compact, and it converges to id$_{L^2(A,\tau)}$ strongly, as in the definition of the Haagerup property.
However, unfortunately, these ranges are not contained in $A$ in general.
So we have to modify $\Phi_\alpha$ to take its values in $A$.
To do this, we need the following notations.
We identify $\pi_\tau(A)''$ with the direct summand $A^{**}c(\pi_\tau)$ of $A^{**}$, where $c(\pi_\tau)$ is the central cover of $\pi_\tau$ \cite[Definition 1.4.2]{BO}.
Denote the image of $\Phi_\alpha$ by $E_\alpha$, and denote the linear span of $1_{A^{**}}$ and $E_\alpha$ by $F_\alpha$,
which is a finite dimensional \Cs -subalgebra of $A^{**}$.
Denote the canonical inclusion $F_\alpha\hookrightarrow A^{**}$ by $\iota_\alpha$.
Since $\iota_\alpha$ is a $\ast$-homomorphism, in particular it is contained in CP($F_\alpha$,~$A^{**}$).
Then, by using the canonical bijective correspondence between CP($F_\alpha$,~$A^{**}$) and $(A^{**}\otimes{F_\alpha})_+$ \cite[Theorem 1.5.12]{BO},
and the density of $A\otimes{F_\alpha}$ in $A^{**}\otimes{F_\alpha}$ in the strong operator topology,
we can find a bounded net $(\Psi_\beta^{(\alpha)})_\beta$
from CP($F_\alpha$,~$A$) that converges to $\iota_\alpha$ in the pointwise strong operator topology, as $\beta$ tends to infinity (by the Kaplansky density theorem).
Then, since each $\Psi_\beta^{(\alpha)}(1)$ is contained in $A\subset A^{**}$ and it converges to $1_{A^{**}}=1_A\in A$ weakly as $\beta$ tends to infinity, by retaking a net of {\rm c.p.\null} maps from the convex hull of $\{\Psi_\beta^{(\alpha)}\}_\beta$,
we may assume $(\Psi_\beta^{(\alpha)}(1))_\beta$ converges to $1$ in norm.
(Recall the Hahn-Banach separation Theorem.)
We remark that each support of $\tau\circ\Psi_\beta^{(\alpha)}|_{E_\alpha}$ is contained in that of $\tau|_{E_\alpha}$, which is equal to $E_\alpha$,
and the former net converges to $\tau|_{E_\alpha}$ pointwise as $\beta$ tends to infinity.
From this and the fact $E_\alpha$ has a finite dimension, we can choose a net $(c_\beta^{(\alpha)})_\beta$ of positive numbers,
such that $(c_\beta^{(\alpha)})_\beta$ converges to $1$ as $\beta$ tends to infinity and
$\tau\circ\Psi_\beta^{(\alpha)}|_{E_\alpha}\leq c_\beta^{(\alpha)}\tau|_{E_\alpha}$.
Put \[d_\beta^{(\alpha)}:=\max\{c_\beta^{(\alpha)},\|\Psi_\beta^{(\alpha)}(1)\|\}\]
for each $\alpha$ and $\beta$.
Then by definition of $d_\beta^{(\alpha)}$, each map $(d_\beta^{(\alpha)})^{-1}\Psi_\beta^{(\alpha)}$ is a {\rm c.c.p.\null} map that decreases $\tau$.
Now it is easy to take the desired net from the set $\{c(\pi_\tau)\Psi_\beta^{(\alpha)}\circ\Phi_\alpha\}_{\alpha,\beta}$,
here we identify $A$ with the \Cs -subalgebra $\pi_\tau(A)$ of $\pi_\tau(A)''=A^{**}c(\pi_\tau)$, not with the canonical \Cs -subalgebra of $A^{**}$.
\end{proof}
Indeed, in the above proof, we only use the injectivity of ${\pi_\tau (A)}''$.
From this, we can also apply the proof of Theorem \ref{Thm:nuc} for some other cases.
Here we summarize the cases Theorem \ref{Thm:nuc} is applicable.
Part (1) is pointed out by Professor Narutaka Ozawa.
\begin{Cor}\label{Cor:nsb}
\begin{enumerate}[\upshape (1)]
\item Let $A$ be a unital exact \Cs -algebra with a faithful amenable tracial state $\tau$.
Then the pair $(A, \tau)$ has the Haagerup property.
\item Let $A$ be a unital residually finite dimensional \Cs -algebra with a faithful tracial state.
Then $A$ has a faithful tracial state $\tau$ with the Haagerup property.
\end{enumerate}
\end{Cor}
\begin{proof}
(1) It suffices to show $\pi_\tau(A)''$ is injective.
By amenability of $\tau$, the product $\ast$-homomorphism
\[\pi_\tau\times\pi^{\rm op}_\tau\colon A\odot A^{\rm op}\rightarrow B(L^2(A, \tau))\]
is continuous with respect to the spatial tensor product norm \cite[Theorem 6.2.7]{BO}.
Then by universality of the double dual, it extends to the normal $\ast$-homomorphism
from $(A\otimes A^{\rm op})^{**}$ to $B(L^2(A, \tau))$.
Notice that since $A$ is exact, it has property $\mathrm{C}''$ \cite[Theorem 9.3.1]{BO}.
(This still holds for non-separable case, by \cite[Proposition 9.2.5 and Lemma 9.2.8]{BO}.)
So the canonical inclusion
\[A^{**}\odot(A^{\rm op})\rightarrow(A\otimes A^{\rm op})^{**}\]
is continuous with respect to the spatial tensor product norm.
Consequently, the product $\ast$-homomorphism
\[\pi_\tau(A)''\odot \pi^{\rm op}_\tau(A^{\rm op})\rightarrow B(L^2(A, \tau))\]
is continuous with respect to the spatial tensor product norm.
Note that $\left(\pi^{\rm op}_\tau(A^{\rm op})\right)'=\pi_\tau(A)''$ \cite[Theorem 6.1.4]{BO}.
Now injectivity of $\pi_\tau(A)''$ follows from Lance's trick \cite[Proposition 3.6.5]{BO}.\\
(2) By assumption, there exists a faithful tracial state $\tau$ on $A$ such that
$\pi _\tau(A)''$ is a type {\rm I} von Neumann algebra.
Since a type {\rm I} von Neumann algebra is injective, we obtain the desired result.
\end{proof}
\begin{Rem}
In the above cases, the approximation maps in the Haagerup property can be taken as finite rank {\rm u.c.p.\null} maps.
This is an interesting phenomenon:
Though \Cs -algebras as above fail to have the internal (or even external) finite dimensional approximation by {\rm u.c.p.\null} maps in the norm topology in general,
they have the internal finite dimensional approximation by {\rm u.c.p.\null} maps in the $L^2$-topology.
\end{Rem}
\begin{Rem}
Part (1) of Corollary \ref{Cor:nsb} is applicable to many exact \Cs -algebras.
For example, any unital simple exact quasi-diagonal \Cs -algebra has a faithful amenable tracial state \cite[Proposition 7.1.16]{BO}.
On the other hand, for a unital \Cs -subalgebra of a nuclear \Cs -algebra $A$ with a faithful tracial state $\tau$,
the restriction $\tau|_B$ of $\tau$ on $B$ is amenable,
since any tracial state on a nuclear \Cs -algebra is amenable
and a restriction of an amenable tracial state is again amenable.
Consequently, all of these \Cs -algebras have a faithful tracial state with the Haagerup property.
\end{Rem}
\begin{Rem}
The proof of Corollary \ref{Cor:nsb} (1) only uses property $\mathrm{C}''$,
which is equivalent to the local reflexivity.
However, the existence of a non-exact locally reflexive \Cs -algebra is a long standing open problem. 
\end{Rem}

 Next we prove the permanence of the Haagerup property under a few canonical constructions in \Cs -algebras. While these results can be obtained adapting Jolissaint's von Neumann algebraic method from [18], for the readers convenience, we give a proof below. First we record the following lemma whose straightforward proof is left to the reader.
\begin{Lem}\label{Lem:con}
Let $A$ be a \Cs -algebra, $\tau$ be a faithful tracial state on $A$, $(A_i)_i$ be an increasing net of \Cs -subalgebras of $A$
whose union is dense in $A$ with respect to the $L^2$-norm determined by $\tau$.
Assume for each $i$, there is a trace-preserving conditional expectation $E_i$ from $A$ onto $A_i$.
Then the pair $(A, \tau)$ has the Haagerup property if and only if
each pair $(A_i, \tau)$ has the Haagerup property.
\end{Lem}

Now we establish the permanence properties of the Haagerup property.
Here we restate Theorem \ref{Thmint:per}.
\begin{Thm}\label{Thm:per}
Let ${(A_i, \tau_i)}_{i\in I}$ be a family of \Cs -algebras with the Haagerup property indexed by a set $I$.
Then the following hold.
\begin{enumerate}[\upshape (1)]
\item If I is countable, then the direct product $(\prod_{i\in I}A_i, \tau)$ has the Haagerup property for any tracial state $\tau$ of the form $\tau=\sum_{i\in I}c_i\tau_i$,
where $(c_i)_{i\in I}$ is a family of positive numbers whose sum is $1$.
More generally, any \Cs -subalgebra of $(\prod_{i\in I}A_i, \tau)$ which contains both $\bigoplus_{i\in I} A_i$ and $1$ has the Haagerup property with respect to the restriction tracial state.
\item The spatial tensor product $(\bigotimes_{i\in I}A_i, \bigotimes_{i\in I}\tau_i)$ has the Haagerup property.
\item The reduced free product $(A, \tau)=$\lower0.25ex\hbox{\LARGE $\ast$}${}_{i\in I}(A_i,\tau_i)$ has the Haagerup property.
\end{enumerate}
\end{Thm}

\begin{proof}
(1) We may assume $I=\mathbb{N}.$
For each $n\in\mathbb{N}$, take an approximation net $(\Phi_{n,j})_{j\in J_n}$ of {\rm u.c.p.\null} maps
of the Haagerup property of $(A_n,\tau _n)$.
Replace $J_n$ by $\prod_k{J_k}$ for each $n\in\mathbb{N}$, we may assume all index sets of the nets are the same one, say $J$.
Then for each $n\in\mathbb{N}$ and $j\in J$,
we define a {\rm u.c.p.\null} map $\Psi_{n,j}$ on $\prod_{n\in\mathbb{N}} A_n$ by
\[\Psi_{n,j}:=\left(\bigoplus_{k\leq n}\Phi_{k,j}\right)\oplus\left(\sum_{k>n}c_k\tau_k\right).\]
Then the net $(\Psi_{n,j})_{n,j}$ satisfies the desired condition.
Moreover, each range of $\Psi_{n,j}$ is contained in the unitization of $\bigoplus_{n\in\mathbb{N}} A_n$.
So the second part of the claim also follows.\\
(2)
By the previous lemma, it suffices to consider the case $I=\{1,2\}$.
Let $(A,\tau), (B,\nu)$ be two pairs of \Cs -algebras and faithful tracial states both of which have the Haagerup property.
By assumption and Remark 2.2, we can choose nets of trace-preserving {\rm u.c.p.\null} maps $(\phi_j)_{j\in J}, (\psi_j)_{j\in J}$
which give the Haagerup property of $(A,\tau), (B,\nu)$ respectively.
Then the net $(\phi_j\otimes\psi_j)_{j\in J}$ obviously gives the Haagerup property of $(A\otimes B, \tau\otimes\nu)$. \\
(3)
Similarly, it suffices to consider the case $I=\{1,2\}$.
With the notations as above, we define
\[\tilde{\phi}_{j, k}:=c_k\phi_j+(1-c_k)\tau, \tilde{\psi}_{j, k}:=c_k\psi_j+(1-c_k)\nu\]
for each $j\in J$ and $k\in\mathbb{N}$, where $c_k=1-1/k$.
(For the notion of reduced free product, we refer the reader to \cite[Section 4.7]{BO}.)
We will show the net $(\tilde{\phi}_{j, k}\ast\tilde{\psi}_{j, k})_{(j, k)\in J\times\mathbb{N}}$ gives the Haagerup property of $(A, \tau)\ast(B, \nu)$.
Note that by definition of $\tilde{\phi}$'s and $\tilde{\psi}$'s, it obviously satisfies the conditions listed on Definition \ref{Def:Haa} excepting the $L^2$-compactness.
For the $L^2$-compactness, notice that the GNS-space of the reduced free product $L^2((A, \tau)\ast(B, \nu))$ is canonically isomorphic to the free product $(L^2(A, \tau), 1_A^\tau)\ast(L^2(B, \nu), 1_B^\nu)$ of GNS-spaces.
Then since the restriction of $\tilde{\phi}_{j, k}$ to
$L^2(A, \tau)^o=L^2(A, \tau)\ominus\mathbb{C}1_A^\tau$ has the norm less than or equal to $c_k$ and similarly for $\tilde{\psi}_{j, k}$,
$\tilde{\phi}_{j, k}\ast\tilde{\psi}_{j, k}$ is a $c_0$-direct sum of compact operators as a bounded operator on $(L^2(A, \tau), 1_A^\tau)\ast(L^2(B, \nu), 1_B^\nu)$ by definition.
Therefore it is a compact operator, as desired.
\end{proof}

Next we study the permanence of the Haagerup property under the reduced crossed product construction.
As we will see in Theorem \ref{Thm:Aff}, this is no longer true in general.
However, in the following AF-setting, we have the permanence property.
This is pointed out by the referee.

\begin{Thm}\label{Thm:per2}
Let $\Gamma$ be a group with the Haagerup property acting on a unital \Cs -algebra $A$.
Assume the following hold.
\begin{itemize}
\item $A$ has a $\Gamma$-invariant faithful tracial state $\tau$.
\item $\Gamma$ is the union of an increasing net $(\Gamma_i)_{i\in I}$ of subgroups.
\item There is an increasing net $(A_i)_{i\in I}$ of finite dimensional \Cs -subalgebras of $A$, 
whose union is dense in $A$ with respect to the $L^2$-norm determined by $\tau$, and each $A_i$ is $\Gamma_i$-invariant.
\end{itemize}
Then the reduced crossed product $A\rtimes _r\Gamma$ has the Haagerup property.
\end{Thm}
\begin{proof}
We will show $A\rtimes _r\Gamma$ has the Haagerup property with respect to the canonical extension $\tilde{\tau}$ of $\tau$.
For each $i\in I$, there exists a unique $\tau$-preserving conditional expectation $E_i$ from $A$ onto $A_i$ \cite[Lemma 1.5.11]{BO}.
Notice that, by uniqueness, it must be $\Gamma_i$-equivariant. 
Hence by \cite[Exercise 4.1.4]{BO}, $E_i$ extends to a conditional expectation from $A\rtimes _r\Gamma_i$ onto
$A_i\rtimes _r\Gamma_i$, which preserves $\tilde{\tau}$ by definition.
At the same time, we also have a $\tilde{\tau}$-preserving conditional expectation from $A\rtimes _r\Gamma$ onto $A\rtimes _r\Gamma_i$. Indeed, first represent $A\rtimes _r\Gamma$ on $L^2(A, \tau)\otimes l^2(\Gamma)$ in the canonical way and similarly for $A\rtimes _r\Gamma_i$.
Consider the conditional expectation on $B(L^2(A, \tau)\otimes l^2(\Gamma))$ induced by the projection $p=1\otimes\chi_{\Gamma_i}$.
Then the restriction of it to $A\rtimes _r\Gamma$ gives the desired conditional expectation.
Composing these two maps, we have a $\tilde{\tau}$-preserving conditional expectation from $A\rtimes _r\Gamma$ onto $A_i\rtimes _r\Gamma_i$.
Note that by Theorem \ref{Thm:Don} (3), each $A_i\rtimes _r\Gamma_i$ has the Haagerup property with respect to $\tilde{\tau}$.
Then, since the union of $A_i\rtimes _r\Gamma_i$'s is dense in $A\rtimes _r\Gamma$ with respect to the $L^2$-topology,
Lemma \ref{Lem:con} completes the proof.
\end{proof}

\begin{Rem}
From Theorem \ref{Thm:Aff}, the corresponding result of Theorem \ref{Thm:per2} in the AH-setting is no longer true, even if the index of nets is singleton.
\end{Rem}

\section{\texorpdfstring{An Application of the Haagerup Property for \Cs -algebras}{An Application of the Haagerup Property for C*-algebras}}\label{sec:App}

In this section, we give an application of the Haagerup property for \Cs -algebras.
Our results rely heavily on techniques in von Neumann algebras which trace back to the work of Popa from \cite{Pop}.
Popa's theorem says the Haagerup property is a strong negation of relative property~{\rm (T)} in the context of von Neumann algebras.
We extend this rigidity theorem to the context of \Cs -algebras.

Our theorem does not depend on the tracial states, therefore it is convenient to introduce the following class of \Cs -algebras.
\begin{Def}
Set $\mathcal{H}$ be the class of all \Cs -algebras which has a faithful tracial state $\tau$ with the Haagerup property.

\end{Def}

By the results in Section \ref{sec:Ex}, the class $\mathcal{H}$ is quite large.
It contains all nuclear \Cs -algebras with a faithful tracial state, many exact \Cs -algebras (for example, unital simple exact quasi-diagonal \Cs -algebras),
residually finite dimensional \Cs -algebras with a faithful tracial state, the reduced group \Cs -algebras of groups with the Haagerup property,
and is closed under taking the direct product, the spatial tensor product and the reduced free product.
However, we need to remark that the class $\mathcal{H}$ is not closed under taking a quotient, even if the quotient has a faithful tracial state.
To see this, consider the full group \Cs -algebra ${\rm C}^*(F_\infty )$ of the free group $F_\infty$ of countably many generators.
Then it is residually finite dimensional by Choi's theorem \cite[Theorem 7]{Choi}.
From this and separability of ${\rm C}^*(F_\infty )$, it is contained in the class $\mathcal{H}$ by Corollary \ref{Cor:nsb} (2).
Note that any unital separable \Cs -algebra arises as a quotient of ${\rm C}^*(F_\infty )$,
and as we soon see in Corollary \ref{Cor:Nem}, there is a unital separable \Cs -algebra which has a faithful tracial state but is not contained in the class $\mathcal{H}$.

The following theorem, due to Popa \cite{Pop}, will play a key role in deriving our application.

\begin{Thm}[Popa {\cite[Theorem 5.4 (1)]{Pop}}]\label{Thm:Popa}
Let $M$ be a von Neumann algebra with a faithful normal tracial state, $B$ a von Neumann subalgebra of $M$.
If $M$ has the Haagerup property and the pair $(M,B)$ has relative property~{\rm (T)},
then $B$ is not diffuse.
\end{Thm}

\noindent Here we need two comments.
\begin{Rem}\label{Rem:Popa1}
In the statement of \cite[Theorem 5.4]{Pop}, it is only considered the case $M$ is a type {\rm I\hspace{-.1em}I$_1$} factor.
However, in his proof, we do not need either the factoriality of $M$ or the assumption $M$ is of type {\rm I\hspace{-.1em}I$_1$},
since \cite[Proposition 4.1]{Pop} is proved for any finite von Neumann algebras with a faithful normal tracial state.
\end{Rem}
\begin{Rem}\label{Rem:Popa2}
If we cut $M$ by the projection corresponding to the diffuse part of $B$,
then the resulting von Neumann algebra still has the Haagerup property and the resulting pair also has relative property~{\rm (T)} \cite[Proposition 4.7 (2)]{Pop}.
From this, if $B$ has a nonzero diffuse direct summand,
then this contradicts Theorem \ref{Thm:Popa}.
Consequently, $B$ must be a direct sum of matrix algebras.
\end{Rem}

We now apply Popa's theorem to the context of \Cs -algebras.
The proofs of the following lemmas are straightforward, so we only give sketches of the proofs.
\begin{Lem}\label{Lem:GNSH}
Let $(A,\tau)$ be a pair of a unital \Cs -algebra and a faithful tracial state on $A$.
Let $\pi_\tau$ be the GNS-representation of $\tau$.
If the pair $(A,\tau)$ has the Haagerup property,
then so does $(\pi _\tau(A)'',\tau)$.
\end{Lem}
\begin{proof}[Sketch of the proof]
Note that any trace-preserving {\rm u.c.p.\null} map on $A$ extends
to a trace-preserving {\rm u.c.p.\null} map on the GNS-closure,
which is $L^2$-compact if the original one is.
The extensions of approximation maps of the Haagerup property for $(A,\tau)$
establish the Haagerup property of $(\pi _\tau(A)'',\tau)$.
\end{proof}

\begin{Lem}\label{Lem:GNST}
Let $A$ be a \Cs -algebra, $B$ be a \Cs -subalgebra of $A$ and $\tau$ be a tracial state on $A$.
If the pair $(A,B)$ has relative property~{\rm (T)} {\rm (}in the sense of Leung-Ng{\rm )}, then the pair $\left(\pi _\tau (A)'',\pi _\tau (B)''\right)$ of GNS-closures has relative property~{\rm (T)} in the sense of Popa.
\end{Lem}
\begin{proof}[Sketch of the proof]
Since the left and right actions of a Hilbert bimodule of a von Neumann algebra $M$ are normal,
for any $\sigma$-strongly dense subset $S$ of $M$, any $S$-central vector of $H$ is indeed $M$-central.
From this, our claim follows easily.
\end{proof}

Now, we obtain the rigidity result, Theorem \ref{Thmint:T}.
\begin{Thm}\label{Thm:T}
Let $A\in\mathcal{H}$,~$B$ be its \Cs -subalgebra. If the pair $(A,B)$ has relative property~{\rm (T)},
then $B$ is residually finite dimensional.
\end{Thm}
\begin{proof}
Choose a faithful tracial state $\tau$ on $A$ with the Haagerup property.
Then by Lemmas \ref{Lem:GNSH} and \ref{Lem:GNST},
the pair $(\pi_\tau(A)'', \tau)$ has the Haagerup property and
the pair $(\pi_\tau(A)'', \pi_\tau(B)'')$ has relative property~{\rm (T)}.
Hence, by Popa's theorem, $\pi_\tau(B)''$ is a direct sum of matrix algebras.
\end{proof}
\begin{Rem}
Combining the proof above with Theorem \ref{Thm:nuc} (and the fact that nuclearity passes to a quotient),
we have a generalization of \cite[Proposition 12]{Bek} as follows.\\
{\it Let $(A, \tau)$ be a pair of a \Cs -algebra and a faithful tracial state on $A$ that has both the Haagerup property and property~{\rm (T)}.
Then $L^2(A, \tau)$ decomposes as a direct sum of finite dimensional $A$-submodules.}
\end{Rem}\

Here we revisit the rigidity theorem of Robertson \cite[Theorem C]{Rob}.

\begin{Cor}\label{Cor:Nem}
Let $\Gamma$ be a property~{\rm (T)} group, $A\in\mathcal{H}$.
Then any unitary representation of $\Gamma$ on $A$ is weakly equivalent
to a direct sum of finite dimensional representations.
In particular, if $\Gamma$ is an infinite property~{\rm (T)} group,
then there is no nonzero $\ast$-homomorphism
from the reduced group \Cs -algebra ${\rm C}_r^*(\Gamma)$ into $A$.
\end{Cor}
\begin{proof}
By Leung-Ng's theorem, the full group \Cs -algebra ${\rm C}^*(\Gamma)$ of $\Gamma$ has property~{\rm (T)}.
Since property~{\rm (T)} passes to a quotient (Proposition \ref{Prop:PerT}), for any representation $\pi$ of $\Gamma$ on $A$,
the \Cs -subalgebra of $A$ generated by the image of $\pi$, which is isomorphic to a quotient of the full group \Cs -algebra of $\Gamma$, has property~{\rm (T)}.
Since it is a \Cs -subalgebra of a \Cs -algebra in the class $\mathcal{H}$,
it is residually finite dimensional by Theorem \ref{Thm:T}.
This proves our first claim.
For the last statement, recall the reduced group \Cs -algebra has a finite dimensional representation if and only if the group is amenable.
\end{proof}
\begin{Rem}\label{Rem:Nem}
By the proof of Theorem \ref{Thm:T} and Corollary \ref{Cor:Nem}, for any faithful tracial state $\tau$ on $A$ with the Haagerup property and a unitary representation of a group $\Gamma$ with property~{\rm (T)} on $A$,
the induced unitary representation of $\Gamma$ on $L^2(A, \tau)$ is equivalent
to a direct sum of finite dimensional representations.
The original statement of Robertson's theorem \cite[Theorem C]{Rob} follows from this.
\end{Rem}
\begin{Rem}
Certainly, the last assertion of Corollary \ref{Cor:Nem} still holds if $\Gamma$ is a group which has a non-amenable subgroup $\Lambda$
such that the pair $(\Gamma, \Lambda)$ has relative property~{\rm (T)}.
However, we do not know a non-trivial example of such group.
That is, the case $\Gamma$ does not contain an infinite property~{\rm (T)} subgroup.
A list of groups which might satisfy the above condition is given in the book \cite[Chapter 7]{BHV}.

\end{Rem}
\begin{Rem}
Gromov \cite{Gro} constructs a property~{\rm (T)} group without nontrivial finite dimensional representations.
(See also \cite[p.553 Remarks (2)]{Rob}.)
If $\Gamma$ is such a group and $A\in\mathcal{H}$,
then by Corollary \ref{Cor:Nem}, there is no nonzero
group-homomorphism from $\Gamma$ into the unitary group $U(A)$ of $A$.
This shows the group structure (without topological information) of the unitary group $U(A)$ of a unital \Cs -algebra $A$
sometimes remembers the information that $A$ is not contained in the class $\mathcal{H}$
(e.g., the case $A={\rm C}_r^*(\Gamma)$ for a group $\Gamma$ as above).
\end{Rem}
\begin{Rem}
The obstruction of the Haagerup property to property~{\rm (T)} given in Theorem \ref{Thm:T} is the best possible form.
There is an infinite dimensional property~{\rm (T)} \Cs -algebra which is contained in the class $\mathcal{H}$.
Indeed, the following holds.
\end{Rem}
\begin{Prop}\label{Prop:RFDT}
If $A$ is a unital \Cs -algebra which is residually finite dimensional with property~{\rm (T)} and a faithful tracial state,
then $A$ is contained in the class $\mathcal{H}$.
\end{Prop}
\noindent Before the proof, we need a comment.
Although this is a special case of Corollary \ref{Cor:nsb} (2), we prefer to give an independent proof, which is much more elementally, by using a result of Brown about property~{\rm (T)} \Cs -algebras from \cite{Bro}.

\begin{proof}
Let $A$ be as above.
Let $\{\pi_i\}_{i\in I}$ be a complete representation system of the set of all equivalent classes of
finite dimensional irreducible representations of $A$.
Then $\bigoplus_{i\in I}\pi_i$ is a faithful representation of $A$ by assumption.
Hence we can regard $A$ as a unital \Cs -subalgebra of $\prod_{i\in I}{\mathbb{M}_{d_i}}$, where $d_i$ is the dimension of $\pi_i$.
Then by the existence of Kazhdan projections \cite[Theorem 3.4]{Bro},
the unit of the $i$th direct summand $1_{\mathbb{M}_{d_i}}$ is contained in $A$ for all $i\in I$.
Then, by irreducibility of $\pi_i$,~$i$th direct summand $\mathbb{M}_{d_i}$ is contained in $A$ for all $i\in I$.
Hence $\bigoplus_{i\in I}{\mathbb{M}_{d_i}}$ is contained in $A$.
Then by the existence of a faithful tracial state, $I$ must be countable.
Hence $A$ is contained in the class $\mathcal{H}$ by Theorem \ref{Thm:per} (1).
\end{proof}
Here we give an infinite dimensional example of \Cs -algebra which has both property (T) and the Haagerup property.
\begin{Exm}\label{Exm:SLn}
Let $n\geq 3.$
On the group algebra $\mathbb{C}[\SL _n({\mathbb{Z}})]$ of $\SL _n({\mathbb{Z}})$, define the \Cs -seminorm $\|\cdot\|_{\rm fin}$ as follows:
\[\| x\|_{\rm{fin}}:=\sup\{\mbox{ }\| \pi(x)\|\mbox{ }|\mbox{ }\pi \mbox{ {\rm is a finite representation of }} \SL _n({\mathbb{Z}})\}.\]
Then define the \Cs -algebra ${\rm C}^*_{\rm fin}(\SL _n({\mathbb{Z}}))$ as the completion of $\mathbb{C}[\SL _n({\mathbb{Z}})]$
with respect to the seminorm $\|\cdot\|_{\rm fin}$.
Since $\SL _n({\mathbb{Z}})$ is residually finite, the left regular representation is weakly contained in a direct sum of finite dimensional representations.
Therefore the seminorm $\|\cdot\|_{\rm{fin}}$ is (strictly) greater than the reduced norm $\|\cdot\|_r$.
Hence this is indeed a norm and consequently ${\rm C}^*_{\rm fin}(\SL _n({\mathbb{Z}}))$ is infinite dimensional.
Moreover, since property~{\rm (T)} passes to a quotient (Proposition \ref{Prop:PerT}), ${\rm C}^*_{\rm fin}(\SL _n({\mathbb{Z}}))$ has property~{\rm (T)}.
On the other hand, since ${\rm C}^*_{\rm fin}(\SL _n({\mathbb{Z}}))$ is residually finite dimensional, it is contained in the class $\mathcal{H}$ by Proposition \ref{Prop:RFDT}.
\end{Exm}
\begin{Rem}
In Bekka's paper \cite{Bek2}, he proves that
the number of the quasi-equivalent classes of the infinite dimensional finite factor representations of $\SL _n({\mathbb{Z}})$
is less than or equal to the cardinality of the center of $\SL _n({\mathbb{Z}})$ \cite[Theorem 3]{Bek2}.
The center of $\SL _n({\mathbb{Z}})$ is $\{ I\}$ if $n$ is odd and is $\{\pm I\}$ if $n$ is even .
If $n$ is odd, then the left regular representation $\lambda$ of $\SL _n({\mathbb{Z}})$ is a finite factor representation
so this is the only infinite dimensional finite factor representation of $\SL _n({\mathbb{Z}})$.
If $n$ is even, put $p:=(\lambda_{I}-\lambda_{-I})/2$,~$q=p^\perp.$
Then both $p$ and $q$ are the central projections of the group von Neumann algebra $L(\SL _n({\mathbb{Z}}))$.
Since the center of $L(\SL _n({\mathbb{Z}}))$ is spanned by $p$ and $q$, both subrepresentations of $\lambda$ reduced by $p$ and $q$ are finite factor representations.
Clearly, these two representations are mutually different: One is faithful but the other is not.
Consequently, in both cases, any infinite dimensional finite factor representation of $\SL _n({\mathbb{Z}})$ is quasi-equivalent to
a subrepresentation of $\lambda$.
(In particular, Bekka's super-rigidity theorem for $\SL _n({\mathbb{Z}})$ \cite[Theorem 1]{Bek2} indeed holds without taking finite index subgroups.)
Since $\lambda$ is weakly contained in a direct sum of finite dimensional representations,
our completion in the previous example is indeed the maximal tracial completion,
namely, it is the maximal completion which makes the completed algebra to have a separating family of tracial states.
So it is a natural object.
Note that again by Bekka's result, this completion does not coincide with the maximal one.
This was raised as a question by Kirchberg \cite[p.487 (P4)]{Kir} and proved by Bekka \cite[Section 8]{Bek2}.
\end{Rem}
\begin{Rem}
For nuclear case, a much stronger negation of property~{\rm (T)} is proved by Brown as follows \cite[Theorem 5.1]{Bro}.\\
{\it Let $A$ be a \Cs -algebra which is nuclear and has property~{\rm (T)}.
Then $A$ is of the form $B\oplus C$,
where $B$ is finite dimensional and $C$ admits no tracial states.
In particular, if we further assume $A$ has a faithful tracial state,
then $A$ must be finite dimensional.}
\end{Rem}
Using Example \ref{Exm:SLn}, it can be also shown the following property of the Haagerup property,
which does not occur in the context of von Neumann algebras \cite[Proposition 2.4.]{Jol}.
\begin{Thm}\label{Thm:trd}
The Haagerup property for \Cs -algebras does depend on the choice of a faithful tracial state.
\end{Thm}
\begin{proof}
Let $A={\rm C}^*_{\rm fin}(\SL _n({\mathbb{Z}}))$, where $n\geq 3$.
We already know it has a faithful tracial state with the Haagerup property.
So to show the claim, it suffices to find a faithful tracial state $\tau$ on $A$ without the Haagerup property.
Remark that, since the left regular representation $\lambda$ of $\SL _n({\mathbb{Z}})$ is weakly contained in a direct sum of finite dimensional representations,
$\delta_e$ extends to a tracial state of $A$, say the extension $\tau_1$.
Define $\tau=(\tau_1+\tau_2)/2$, where $\tau_2$ is an arbitrary faithful tracial state on $A$.
We will show the pair $(A, \tau)$ does not have the Haagerup property.
Assume by contradiction that $(A, \tau)$ has the Haagerup property, i.e., there exists a sequence $(\Phi_k)_k$ of {\rm u.c.p.\null} maps on $A$ satisfying the properties listed on Definition \ref{Def:Haa}.
Consider $l_{\tau}^2(\SL _n({\mathbb{Z}}))$, which is the GNS-space of $\tau$.
For any $f\in c_c(\SL _n({\mathbb{Z}}))$, we have $\|f\|_2^2\leq 2\|f\|_\tau^2$,
hence the identity map on $c_c(\SL _n({\mathbb{Z}}))$ extends to a bounded operator from $l_{\tau}^2(\SL _n({\mathbb{Z}}))$ into $l^2(\SL _n({\mathbb{Z}}))$.
Denote the extension by $\pi$.
Now for each $k$, we define a complex valued function $\psi_k$ on $\SL _n({\mathbb{Z}})$ by
\[\psi_k(g):=\langle\delta_g, \pi(\Phi_k(g)\delta_e^\tau)\rangle_2=\langle\delta_g, \lambda\left(\Phi_k(g)\right)\delta_e\rangle_2.\]
Then $\psi_k$ converges to $1$ pointwise as $k$ tends to infinity.
On the other hand, since $\lambda\circ\Phi_k$ is {\rm u.c.p.\null}, $\psi_k$ is positive definite.
Hence the convergence of $\psi_k$ is indeed uniform (since $\SL _n({\mathbb{Z}})$ has property~{\rm (T)}).
However, note that the family $(\delta_g)_{g\in\SL _n({\mathbb{Z}})}$ is an orthonormal basis of $l^2(\SL _n({\mathbb{Z}}))$,
whereas the set $\{\pi(\Phi_k(g)\delta_e^\tau)\}_{g\in\SL _n({\mathbb{Z}})}$ is relatively compact in $l^2(\SL _n({\mathbb{Z}}))$ by the $L^2$-compactness of $\Phi_k$ and the boundedness of $\pi$, which is a contradiction.
\end{proof}

In the context of von Neumann algebras, the non-embeddable result of Corollary \ref{Cor:Nem} still holds for the group von Neumann algebra of a group which has relative property~{\rm (T)}
with respect to an infinite subgroup. This is because the group von Neumann algebra of an infinite group is
always diffuse. The corresponding result is not true in the context of \Cs -algebras, because
the reduced group \Cs -algebra of an infinite group can be residually finite dimensional.
Indeed, many typical relative property~{\rm (T)} groups fail to have the rigidity property.

\begin{Lem}\label{Lem:SL2}
Let $A$ be a unital \Cs -algebra with an action of a group $\Gamma$ with the Haagerup property.
Assume $A$ admits a countable family $(\pi _n)_n$ of $\Gamma$-equivariant finite dimensional representations which separates the points of $A$.
Then the reduced crossed product $A\rtimes _r\Gamma$ embeds into a \Cs -algebra in the class $\mathcal{H}$.
\end{Lem}
\begin{proof}
Take a countable separating family $(\pi _n)_n$ of $\Gamma$-equivariant finite dimensional representations.
Then we have a $\Gamma$-equivariant embedding
\[\bigoplus_n\pi_n\colon A \hookrightarrow \prod _nA/{\ker\pi _n}.\]
By taking the reduced crossed products,
we have an embedding
\[ A\rtimes _r \Gamma \hookrightarrow \prod _n\left(\left(A/{\ker\pi _n}\right)\rtimes _r\Gamma\right). \]
Since each $A/{\ker\pi _n}$ is finite dimensional and $\Gamma$ has the Haagerup property,
the range of the above map is contained in the class $\mathcal{H}$.
\end{proof}
Lemma \ref{Lem:SL2} can apply to many reduced group \Cs -algebras of groups without the Haagerup property.
Here we recall the examples of groups which have relative property~{\rm (T)}.
\begin{Def}
Let $\mathbb{K}$ be an algebraic number field (i.e., a finite extension of the rational number field $\mathbb{Q}$),
$R$ be the ring of integers of $\mathbb{K}$ (i.e., the ring of all elements of $\mathbb{K}$ which are roots of a nonzero monic polynomial with the integer coefficients).
The {\it three-dimensional Heisenberg group with the coefficients in} $R$, denoted by $H_3(R)$, is the subgroup of $\SL _3(R)$
which consists of all upper triangular matrices with the diagonal entries $1$.
Equivalently, $H_3(R)$ is defined  as the set $R^2\times R$
with the group operation \[(x,\lambda )(y,\mu ):=(x+y,\lambda +\mu +\omega (x,y)),\]
where $\omega (x,y):=x_1y_2-x_2y_1$ is the symplectic form.
In the latter picture, $\SL _2(R)$ canonically acts on $H_3(R)$ by acting on the first coordinate.
(Since $\SL _2(R)$ preserves $\omega$, this indeed defines an action on the group $H_3(R)$.)
\end{Def}
\begin{Prop}[\cite{BHV}, \cite{Che}]
Let $\mathbb{K}$ be an algebraic number field,
$R$ be the ring of integers of $\mathbb{K}$.
Then the following hold.
\begin{enumerate}[\upshape (1)]
\item The pair $(R^2\rtimes\SL _2(R), R^2)$ has relative property~{\rm (T)}.
\item The pair $(H_3(R)\rtimes\SL _2(R), H_3(\mathbb{Z}))$ has relative property~{\rm (T)}.
\item The group $\SL _2(R)$ has the Haagerup property.
\end{enumerate}
\end{Prop}

\begin{Thm}\label{Thm:SL2}
Let $\mathbb{K}$ be an algebraic number field,
$R$ be the ring of integers of $\mathbb{K}$.
Then the reduced group \Cs -algebras of $R^2\rtimes\SL _2(R)$,
$H_3(R)\rtimes\SL_2(R)$ embed into \Cs -algebras in the class $\mathcal{H}$.
\end{Thm}
\begin{proof}
First note that since both $R^2$ and $H_3(R)$ are amenable, the full and reduced group \Cs -algebras of these groups are equal.
Since $R$ is finitely generated as an additive group,
for any natural number $n\in \mathbb{N}$,~$R/nR$ is a finite ring.
So $(R/nR)^2$ and $H_3(R/nR)$ are also finite.
Moreover, it is obvious that these quotients are $\SL _2(R)$-equivariant.
Consequently, we obtain $\SL _2(R)$-equivariant finite dimensional representations
\[\pi_n\colon{\rm C}_r^*(R^2)\rightarrow {\rm C}_r^*((R/nR)^2),\]
\[\sigma_n\colon{\rm C}_r^*(H_3(R))\rightarrow {\rm C}_r^*(H_3(R/nR)).\]
Now it is easy to check these families separate points.
Consequently, we can apply Lemma \ref{Lem:SL2} to the \Cs -algebras we have considered.
\end{proof}
\begin{Rem}\label{Rem:SL2}
The same proof also works for the group $\mathbb{F}_p[t]^2\rtimes\SL _2(\mathbb{F}_p[t])$,
which has relative property~{\rm (T)} with the subgroup $\mathbb{F}_p[t]^2$,
by replacing $nR$ by $t^n\mathbb{F}_p[t]$ in the proof.
\end{Rem}

\section{A Rigidity Property of the Affine Groups of the Affine Planes}\label{sec:Aff}
In Section \ref{sec:App}, we have seen that, unlike the case of von Neumann algebras, the non-embeddable theorem for the reduced group \Cs -algebras of relative property~{\rm (T)} groups
fails in general.
The difficulty comes from the fact \Cs -algebras admit many ``mutually singular'' faithful tracial states.
However, if we overcome this difficulty, then we can prove a rigidity theorem for a group, even if the group has no infinite property~{\rm (T)} subgroups.
For instance, we will consider the two classes of groups.
The first class consists of groups with Powers's property.
For a Powers group without the Haagerup property, the non-embeddable theorem follows from the uniqueness of the tracial state on the reduced group \Cs -algebra \cite[Proposition 3]{Har}.
By using the free product, it is easy to construct an artificial group in this class without both the Haagerup property and infinite property~{\rm (T)} subgroups
(e.g., $\left(\mathbb{Z}^2\rtimes\SL _2(\mathbb{Z})\right)\ast\mathbb{Z}$).
However, the author does not know an example of group as above which naturally arises.
So we also study the other class.
These groups do not contain infinite property~{\rm (T)} subgroups and naturally arise in many fields:
Namely, we study a rigidity property of the reduced group \Cs -algebras of the affine groups $\mathbb{K} ^2\rtimes {\rm GL} _2(\mathbb{K})$ of the affine planes
(or more strongly for the subgroup $\mathbb{K} ^2\rtimes\SL _2(\mathbb{K})$) over the fields $\mathbb{K}$.
Note that the group $\mathbb{K} ^2\rtimes {\rm GL} _2(\mathbb{K})$ is the automorphism group of the affine plane over $\mathbb{K}$,
so this is a very natural object.
First we remark that a rigidity property obviously fails when $\mathbb{K}$ is an algebraic extension of a finite field.
In this case, $\mathbb{K}$ is an increasing union of finite subfields.
From this, the affine group over $\mathbb{K}$ is an increasing union of finite subgroups,
in particular it is amenable.
We will show that excepting these amenable cases,
the affine groups have a rigidity property.

\begin{Thm}\label{Thm:Aff}
Let $\mathbb{K}$ be a field which is not an algebraic extension of a finite field.
Then ${\rm C}_r^*(\mathbb{K}^2\rtimes\SL_2(\mathbb{K}))$ cannot be embedded into any \Cs -algebra in the class $\mathcal{H}$.
\end{Thm}
\begin{proof}
We divide the proof into two cases:\\
Case 1: $\mathbb{K}$ has characteristic zero.\\
Assume by contradiction ${\rm C}_r^*(\mathbb{K}^2\rtimes\SL_2(\mathbb{K}))$ embeds into a \Cs -algebra contained in the class $\mathcal{H}$.
Then since the Haagerup property passes to the GNS-closure,
and the Haagerup property passes to von Neumann subalgebras \cite{Jol},
we have a faithful tracial state $\tau$ on ${\rm C}_r^*(\mathbb{K}^2\rtimes\SL_2(\mathbb{K}))$
such that the GNS-closure $\pi_\tau({\rm C}_r^*(\mathbb{K}^2\rtimes\SL_2(\mathbb{K})))''$ has the Haagerup property.
We study the tracial state $\tau$.
Consider the restriction $\tau_0$ of $\tau$ to ${\rm C}_r^*(\mathbb{K}^2)$.
Then $\tau_0$ is a $\SL_2(\mathbb{K})$-invariant tracial state on ${\rm C}_r^*(\mathbb{K}^2)$.
Since the action of $\SL_2(\mathbb{K})$ on $\mathbb{K}^2\setminus\{ 0\}$ is transitive,
$\tau_0$ must be of the form
\[\tau_0=c\chi_{\{ 0\}}+(1-c)\chi_{\mathbb{K}^2},\]
where $c\in[0,1]$.
(Here we identify a state on the reduced group \Cs -algebra ${\rm C}_r^*(\Gamma)$
with the restriction of it to the group $\Gamma$, which is a positive definite function of $\Gamma$.)
By faithfulness of $\tau_0$, we further have $c>0$.
From this, the restriction $\tau_1$ of $\tau$ to ${\rm C}_r^*(\mathbb{Z}^2)$ is of the form
\[\tau_1=c\chi_{\{ 0\}}+(1-c)\chi_{\mathbb{Z}^2}\]
with $c>0$.
From this form, the GNS-closure $\pi_\tau({\rm C}_r^*(\mathbb{Z}^2))''$ of ${\rm C}_r^*(\mathbb{Z}^2)$
has the diffuse direct summand $L(\mathbb{Z}^2)$.
On the other hand, since the pair
\[\left(\mathbb{K}^2\rtimes\SL_2(\mathbb{K}), \mathbb{Z}^2\right)\] has relative property~{\rm (T)},
the pair
\[({\rm C}_r^*(\mathbb{K}^2\rtimes\SL_2(\mathbb{K})), {\rm C}_r^*(\mathbb{Z}^2))\] also has relative property~{\rm (T)}.
Then by taking GNS-closures, we further have the pair
\[\left(\left(\pi_\tau({\rm C}_r^*(\mathbb{K}^2\rtimes\SL_2(\mathbb{K})))\right)'', \left(\pi_\tau({\rm C}_r^*(\mathbb{Z}^2))\right)''\right)\]
has relative property~{\rm (T)}.
Then notice that $\left(\pi_\tau({\rm C}_r^*(\mathbb{Z}^2))\right)''$ has a nonzero diffuse direct summand,
whereas $\left(\pi_\tau({\rm C}_r^*(\mathbb{K}^2\rtimes\SL_2(\mathbb{K})))\right)''$ has the Haagerup property.
This contradicts Theorem \ref{Thm:Popa} and Remark \ref{Rem:Popa2}.\\
Case 2: $\mathbb{K}$ has characteristic $p$.\\
This case is also proved by the same method as in Case 1.
Take a transcendental element $\pi$ over the prime field $\mathbb{F}_p$.
Then, notice that the ring $\mathbb{F}_p[\pi]$ is isomorphic to the polynomial ring over $\mathbb{F}_p$.
Therefore the pair
\[\left(\mathbb{K}^2\rtimes\SL_2(\mathbb{K}), \mathbb{F}_p[\pi]^2\right)\]
has relative property~{\rm (T)}.
Now the same proof as in Case 1 works with $\mathbb{F}_p[\pi]$ plays the same role as $\mathbb{Z}$.
\end{proof}
\begin{Rem}
Guentner-Higson-Weinberger \cite{GHW} show the group $\SL_2(\mathbb{K})$ has the Haagerup property for any field $\mathbb{K}$, as a discrete group.
From this, the von Neumann algebra
$L(\mathbb{K}^2\rtimes\SL_2(\mathbb{K}))$ has the relative Haagerup property with respect to the type {\rm I} von Neumann subalgebra $L(\mathbb{K}^2)$ in the sense of Popa \cite{Pop}.
Then Popa's theorem \cite[Theorem 5.4 (2)]{Pop} (and \cite[Proposition 4.7 (2)]{Pop}) shows any von Neumann subalgebra of $L(\mathbb{K}^2\rtimes\SL_2(\mathbb{K}))$ with property~{\rm (T)}
is of type I.
Consequently, we have any \Cs -subalgebra of ${\rm C}_r^*(\mathbb{K}^2\rtimes\SL_2(\mathbb{K}))$ with property~{\rm (T)} is residually finite dimensional.
That is, any property~{\rm (T)} \Cs -subalgebra of ${\rm C}_r^*(\mathbb{K}^2\rtimes\SL_2(\mathbb{K}))$ does not say anything
in our rigidity theorem Theorem \ref{Thm:T}.
However, these \Cs -algebras have a rigidity property relative to the class $\mathcal{H}$.
This in particular shows the class $\mathcal{H}$ is strictly larger than the complement of the class of \Cs -algebras containing a nontrivial property~{\rm (T)} \Cs -subalgebra.
\end{Rem}
\begin{Rem}\label{Rem:Cross}
From Theorem \ref{Thm:Aff},
the class $\mathcal{H}$ is not closed under taking the reduced crossed product by
a group with the Haagerup property
even if the resulting algebra has a faithful tracial state.
(As we have seen in Section \ref{sec:App}, this is not so obvious.)
\end{Rem}
\begin{Rem}
From Theorems \ref{Thm:SL2} and \ref{Thm:Aff},
we obtain the reduced group \Cs -algebra of $\mathbb{Q}^2\rtimes\SL_2(\mathbb{Q})$
cannot embed into that of $\mathbb{Z}^2\rtimes\SL _2(\mathbb{Z})$.
The corresponding result in the context of von Neumann algebras is not known.
\end{Rem}

\section{\texorpdfstring{Appendix : A Short Proof of the Rigidity Theorem for Reduced Group \Cs -algebras}{Appendix : A Short Proof of the Rigidity Theorem for Reduced Group C*-algebras}}
In this section, we give a short proof of the rigidity theorem for the reduced group \Cs -algebras of property~{\rm (T)} groups.
This only uses a group theoretical argument, and does not use any technique of von Neumann algebras.
Since one of main interesting objects of the rigidity theorem is the reduced group \Cs -algebras of the property~{\rm (T)} groups, this short proof is of independent interest.

We first prepare a lemma for non-amenable groups.
Here we introduce the following notation.
For two positive definite functions $\phi, \psi$ on $\Gamma$,
we say $\phi$ is {\it weakly contained} in $\psi$ if $\pi_\phi$ is weakly contained in $\pi_\psi$,
and denote by $\phi\prec\psi$.

\begin{Lem}\label{Lem:name}
Let $\Gamma$ be a non-amenable group.
Let $\phi$ be a positive definite function on $\Gamma$ that is weakly contained in $\delta_e$.
Then there exists a sequence $(g_n)_n$ of $\Gamma$ such that
its canonical image $(\delta_{g_n}^\phi)_n$ in $l^2_\phi(\Gamma)$ converges to $0$ weakly.
\end{Lem}
\begin{proof}
To prove the lemma, it suffices to show the following.
For any finite subset $F$ of $\Gamma$, $\epsilon >0$, 
there is a $g\in\Gamma$ such that
$|\langle \delta _g^\phi, \delta _f^\phi\rangle_\phi |<\epsilon$ for all $f\in F$.
Notice that if $\phi\prec \delta _e$, then we also have $\overline{\phi}\prec\delta _e$.
Then by Fell's absorption theorem \cite[Theorem 2.5.5]{BO}, we also have $|\phi|^2\prec\delta _e$.
Since $\langle \delta _g^{|\phi|^2}, \delta _h^{|\phi|^2}\rangle_{|\phi|^2} =|\langle \delta _g^\phi,\delta _h^\phi\rangle_\phi|^2$ for all $g,h\in\Gamma$,
it suffices to show the lemma for the case $\phi$ takes nonnegative values.
For such $\phi$, it suffices to show the following claim.
For any finite subset $F$ of $\Gamma$ and $\epsilon >0$, there is a $g\in\Gamma$ such that
$\sum_{f\in F}\langle \delta _g^\phi, \delta _f^\phi\rangle_\phi <\epsilon$.
Remark that all summands are nonnegative, so this condition is sufficient.
Assume this is not true. Then for some finite subset $F$ of $\Gamma$, there is a positive number $c$, such that
for all $g\in\Gamma$, we have $\sum_{f\in F}\langle \delta _g^\phi,\delta _f^\phi\rangle_\phi \geq c$.
From this, the canonical image $\Gamma^\phi:=\{\delta_g^\phi\}_{g\in\Gamma}$ of $\Gamma$ in $l^2_\phi(\Gamma)$ is contained in the closed convex subset
$\{\xi\in l^2_\phi(\Gamma)|\sum_{f\in F}\langle \xi,\delta _f^\phi\rangle_\phi\geq c\}$ of $l^2_\phi(\Gamma)$, which obviously does not contain $0$.
From this, the circumcenter $\xi$ of $\Gamma^\phi$ is a nonzero vector.
Since $\Gamma^\phi$ is a $\Gamma$-invariant subset, $\xi$ must be a $\Gamma$-invariant vector.
This means $1_\Gamma\prec\delta_e$, contrary to the non-amenability assumption of $\Gamma$.
\end{proof} 
\begin{Rem}
The converse of Lemma \ref{Lem:name} is also true.
Recall that $\Gamma$ is amenable if and only if $1_\Gamma$ is weakly contained in $\delta_e$.
Hence this gives a characterization of (non-)amenability of discrete groups.
\end{Rem}
Now we prove the rigidity theorem for the reduced group \Cs -algebras of property~{\rm (T)} groups.
\begin{Thm}
Let $\Gamma$ be an infinite group with property~{\rm (T)}, let $A\in\mathcal{H}$.
Then there is no nonzero $\ast$-homomorphism from the reduced group \Cs -algebra ${\rm C}_r^*(\Gamma)$ to $A$.
\end{Thm}
\begin{proof}
Proceeding by contradiction, assume we have a nonzero $\ast$-homomorphism
$\pi\colon{\rm C}_r^*(\Gamma)\rightarrow A$ with $A\in\mathcal{H}$.
Replacing $A$ by $A_{\pi(1)}$, we may assume $\pi$ is unital.
Let $\tau$ be a faithful tracial state on $A$ with the Haagerup property.
Take an approximation net $(\Phi_n)_n$ of the Haagerup property for $(A,\tau)$.
For each $n$, define a positive definite function $\phi_n$ on $\Gamma$ by
\[\phi_n(g):=\tau(\pi(g)^\ast\Phi_n(\pi(g)))=\langle\Phi_n(\pi(g))\delta_e^\tau, \delta_g^\tau\rangle_\tau,\]
here we simply denote $\tau\circ\pi$ by $\tau$ for notational convenience.
Then $\phi_n$ converges to $1$ pointwise.
Since $\Gamma$ has property~{\rm (T)}, this convergence is indeed uniform.
On the other hand, by Lemma \ref{Lem:name} there is a sequence $(g_k)_k$ of $\Gamma$
whose canonical image in $l^2_\tau(\Gamma)$ converges to $0$ weakly.
This with the $L^2$-compactness of $\Phi_n$'s implies for each $n$,
$\phi_n(g_k)$ converges to zero as $k$ tends to infinity.
This contradicts to the uniform convergence of $(\phi_n)_n$.
\end{proof}

\section{Further Questions}
Here we state and comment on some open questions which arise naturally from our investigation.
\begin{Ques}
Can we recover the information whether the discrete group $\Gamma$ has the Haagerup property or not
from the reduced group \Cs -algebra ${\rm C}_r^*(\Gamma)$?
\end{Ques}
For some particular groups, this question is obviously ``Yes''.
For example, consider two extreme cases. If $\Gamma$ is amenable or has property~{\rm (T)},
then thanks to the result of Lance \cite[Theorem 4.2]{Lan}, Bekka \cite[Theorem 7]{Bek}, respectively, it is true.
It is also true by the result of Dong if the reduced group \Cs -algebra ${\rm C}_r^*(\Gamma)$ has the unique tracial state.
It is known that many important discrete groups satisfy the unique trace condition.
For example, outer automorphism groups of (non-commutative) free groups, torsion-free non-elementary hyperbolic groups, irreducible Coxeter groups, and mapping class groups with trivial center satisfy the unique trace condition.
(For the detail, see \cite{Har2} and references therein.) 
But in general, it seems hard to recover the information about the Haagerup property from the reduced group \Cs -algebra.
Our results Example \ref{Exm:SLn}, Theorems \ref{Thm:trd} and \ref{Thm:SL2} suggest ${\rm C}_r^*(\Gamma)$ may be contained in the class $\mathcal{H}$,
even if the group $\Gamma$ does not have the Haagerup property.

The next question is about a permanence property of the Haagerup property.
\begin{Ques}\label{Ques:sub}
Does the Haagerup property pass to a \Cs -subalgebra?
That is, let $(A,\tau)$ be a pair of a \Cs -algebra and a faithful tracial state on $A$ with the Haagerup property, $B$ be a \Cs -subalgebra
of $A$. Then does the pair $(B,\tau|_B)$ have the Haagerup property?
\end{Ques}
Note first that this is true if $A$ is nuclear.
See Corollary \ref{Cor:nsb}.
Note also that if this is true, then Theorem \ref{Thm:nuc} immediately follows.
Since our proof of Theorem \ref{Thm:nuc} is already complicated,
if this is true, then a proof would be perhaps hard.

Note also Question \ref{Ques:sub} has a positive answer in the context of the von Neumann algebras \cite[Theorem 2.3 (i),(ii)]{Jol}.
The reason we can prove this for the von Neumann algebras is that
we can always construct a trace-preserving conditional expectation \cite[Lemma 1.5.11]{BO}.
But in the context of the \Cs -algebras, we cannot construct a conditional expectation in general,
even if we do not consider the condition about the trace.
For example, let $A$ be a nuclear \Cs -algebra, $B$ be a \Cs -subalgebra of $A$ which is not nuclear.
Then there is no conditional expectation from $A$ onto $B$, because any range of a conditional expectation on a nuclear \Cs -algebra is nuclear.
Note that by Blackadar's theorem \cite[Theorem 1]{Bla}, in the separable case, such a \Cs -subalgebra exists if and only if $A$ is not of type {\rm I}.
The condition that $A$ is of type {\rm I} is quite strong. For example, any (infinite dimensional) UHF-algebra is not of type {\rm I}.

\noindent{\bf Acknowledgements.}
The author would like to thank Professor Yasuyuki Kawahigashi, who is his supervisor,
for his encouragement and advice.
He also would like to thank Professor Narutaka Ozawa
for his valuable comments and suggestions.
He is grateful to thank the referee for his or her careful reading and valuable suggestions.
He is supported by Research Fellow of the Japan Society for the Promotion of
Science and Leading Graduate Course for Frontiers of Mathematical Sciences and Physics.
 

\begin{thebibliography}{99}
\bibitem[Bek06]{Bek}
B. Bekka,
{\it Property~{\rm (T)} for \Cs -algebras.}
Bull. London Math. Soc. {\bf 38} (2006), no. 5, 857--867
\bibitem[Bek07]{Bek2}
B. Bekka, {\it Operator-algebraic superrigidity for $\SL_n({\mathbb{Z}})$, $n\geq3$ .} Invent. Math. {\bf 169} (2007), no. 2, 401--425. 
\bibitem[BHV08]{BHV}
B. Bekka, P. de la Harpe, A. Valette,
 {\it Kazhdan's property~{\rm (T)}.} 
{\it New Mathematical Monographs}, {\bf 11}. Cambridge University Press, Cambridge, 2008. xiv+472 pp.
\bibitem[Bla85]{Bla}
B. Blackadar,
 {\it Nonnuclear subalgebras of \Cs -algebras.} 
J. Operator Theory {\bf 14} (1985), no. 2, 347--350. 
\bibitem[Bro06]{Bro}
N.P. Brown,
{\it Kazhdan's property T and \Cs -algebras.}
J. Funct. Anal. {\bf 240} (2006), no. 1, 290--296. 
\bibitem[BO08]{BO}
N.P. Brown, N. Ozawa,
 {\it \Cs -algebras and finite-dimensional approximations.} 
Graduate Studies in Mathematics, {\bf 88}. American Mathematical Society, Providence, RI, 2008. xvi+509 pp.
\bibitem[CCJJV01]{Che}
P.-A. Cherix, M. Cowling, P. Jolissaint, P. Julg, A. Valette,
 {\it Groups with the Haagerup property.
Gromov's a-T-menability.} Progress in Mathematics, {\bf 197}. Birkhauser Verlag, Basel, 2001. viii+126 pp
\bibitem[Cho83]{Cho}
M. Choda,
 {\it Group factors of the Haagerup type.} 
Proc. Japan Acad. Ser. A Math. Sci. {\bf 59} (1983), no. 5, 174--177. 
\bibitem[Cho80]{Choi}
M.D. Choi,
 {\it The full \Cs -algebra of the free group on two generators.}
Pacific J. Math. {\bf 87} (1980), no. 1, 41--48. 
\bibitem[Con76]{Con}
A. Connes,
{\it Classification of injective factors. Cases {\rm I\hspace{-.1em}I$_1$}, {\rm I\hspace{-.1em}I$_\infty$}, {\rm I\hspace{-.1em}I\hspace{-.1em}I$_\lambda$}, $\lambda\neq 1$.}
Ann. of Math. (2) {\bf 104} (1976), no. 1, 73--115. 
\bibitem[CJ85]{CJ}
A. Connes, V. Jones, {\it Property T for von Neumann algebras.}
Bull. London Math. Soc. {\bf 17} (1985), no. 1, 57--62. 
\bibitem[Don11]{Don}
Z. Dong,
 {\it Haagerup property for \Cs -algebras.}
 J. Math. Anal. Appl. {\bf 377} (2011), no. 2, 631--644.
\bibitem[Gro87]{Gro}
M. Gromov,
 {\it Hyperbolic groups. Essays in group theory}, 75--263, 
Math. Sci. Res. Inst. Publ., {\bf 8}, Springer, New York, 1987
\bibitem[GHW05]{GHW}
E. Guentner, H. Higson, S. Weinberger,
{\it The Novikov conjecture for linear groups.}
Publ. Math. Inst. Hautes Etudes Sci. No. {\bf 101} (2005), 243--268. 
\bibitem[Haa79]{Haa}
U. Haagerup,
{\it An example of a nonnuclear \Cs -algebra which has the metric approximation property.}
Invent. Math. {\bf 50} (1979), no. 3, 279--293.
\bibitem[Har85]{Har}
P. de la Harpe,
{\it Reduced \Cs -algebras of discrete groups which are simple with a unique trace.} 
Lecture Notes in Math. {\bf 1132}, Springer, Berlin, (1985), 230--253,
\bibitem[Har07]{Har2}
P. de la Harpe,
{\it On simplicity of reduced \Cs -algebras of groups.}
 Bull. Lond. Math. Soc. {\bf 39} (2007), no. 1, 1--26
\bibitem[Jol02]{Jol}
P. Jolissaint,
{\it Haagerup approximation property for finite von Neumann algebras.}
 J. Operator Theory {\bf 48} (2002), no. 3, suppl., 549--571
\bibitem[Kir93]{Kir}
E. Kirchberg,
{\it On non-split extensions, tensor products and exactness of group \Cs -algebras.}
Invent. Math. {\bf 112} (1993) 449--489
\bibitem[Lan73]{Lan}
E. C. Lance, {\it On nuclear \Cs -algebras.}
J. Funct. Anal. {\bf 12} (1973), 157--176.
\bibitem[LN09]{LN}
C.W. Leung, C.K. Ng,
{\it Property~{\rm (T)} and strong property~{\rm (T)} for unital \Cs -algebras.}
J. Funct. Anal. {\bf 256} (2009), no. 9, 3055--3070. 
\bibitem[Pop06]{Pop}
S. Popa, {\it On a class of type {\rm I\hspace{-.1em}I$_1$} factors with Betti numbers invariants.}
Ann. of Math. (2) {\bf 163} (2006), no. 3, 809--899.
\bibitem[Rob93]{Rob}
G. Robertson, {\it Property~{\rm (T)} for {\rm I\hspace{-.1em}I$_1$} factors and unitary representations of Kazhdan groups.}
 Math. Ann. {\bf 296} (1993), no. 3, 547--555 
\end{thebibliography}
\end{document}